\documentclass[a4paper, 12pt, reqno]{amsart}
\usepackage[margin=1in]{geometry}
\usepackage{mathrsfs}
\usepackage{mathtools}
\usepackage{array}
\usepackage{tikz, float, tikzscale}
\usepackage{pgfplots}
\pgfplotsset{compat=1.18}
\definecolor{mycolour}{rgb}{0.7,0.7,0.7}

\newtheorem{thm}{Theorem}
\newtheorem*{thm*}{Theorem}
\newtheorem*{conj*}{Conjecture}
\newtheorem{prop}{Proposition}

\usepackage[T1]{fontenc}
\usepackage{libertinus,libertinust1math}

\title[L\"uroth Expansions in Diophantine Approximation]{L\"uroth Expansions in Diophantine Approximation: \\ Metric Properties and Conjectures}
\author{Ying Wai Lee}
\date{\today}

\begin{document}
\begin{abstract}
    This paper focuses on the metric properties of L\"uroth well approximable numbers, studying analogous of classical results, namely the Khintchine Theorem, the Jarn\'ik--Besicovitch Theorem, and the result of Dodson. A supplementary proof is provided for a measure-theoretic statement originally proposed by Tan--Zhou. The Beresnevich--Velani Mass Transference Principle is applied to extend a dimensional result of Cao--Wu--Zhang. A counterexample is constructed, leading to a revision of a conjecture by Tan--Zhou concerning dimension, along with a partial result.
\end{abstract}

\maketitle

\section{Introduction}
The metric theory of Diophantine approximation studies the Lebesgue measure and Hausdorff dimension of subsets of real numbers satisfying specific approximation properties.
Let $\psi:\mathbb{N}\to[0,1/2]$. Define $W(\psi)$, the set of all $\psi$-well approximable numbers, as
\begin{align}
\label{eq: W psi}
    W(\psi)
    \coloneqq\limsup_{q\to+\infty}\bigcup_{p\in\mathbb{Z}}\left\{x\in[0,1):\left|x-\frac{p}{q}\right|<\frac{\psi(q)}{q}\right\}.
\end{align}
Equivalently, a number $x\in[0,1)$ is $\psi$-well approximable if and only if there exist infinitely many $q\in\mathbb{N}$ such that for some $p\in\mathbb{Z}$,
\begin{align}
    \label{eq: Khinchine ineq}
    \left|x-\frac{p}{q}\right|<\frac{\psi(q)}{q}.
\end{align}

The Khintchine Theorem is one of the fundamental results in the subject, and serves as a starting point for various research. The theorem states that if $\psi$ is non-increasing, then the Lebesgue measure of the set $W(\psi)$ satisfies the zero-one law, and is determined by the convergence or divergence of the series $\sum_{q=1}^\infty\psi(q)$. Specifically, if the series converges, then almost every $x\in[0,1)$ is not $\psi$-well approximable; conversely, if the series diverges, then almost every $x\in[0,1)$ is $\psi$-well approximable.
\begin{thm*}[Khinchine \cite{Khintchine1924}, 1924]
Let $\psi:\mathbb{N}\to[0,1/2]$. Suppose $\psi$ is non-increasing. Then
\begin{align*}
    \mathcal{L}(W(\psi))
    =
    \begin{dcases}
        0,& \text{if}\quad\sum\nolimits_{q=1}^\infty \psi(q)<+\infty; \\
        1,& \text{if}\quad\sum\nolimits_{q=1}^\infty \psi(q)=+\infty; \\
    \end{dcases}
\end{align*}
where $\mathcal{L}$ denotes the Lebesgue measure on $\mathbb{R}$.
\end{thm*}

Let $\tau\geq0$. Define $W(\tau)$, the set of all $\tau$-well approximable numbers, as
\begin{align}
\label{eq: W tau}
    W(\tau)
    \coloneqq\limsup_{q\to+\infty}\bigcup_{p\in\mathbb{Z}}\left\{x\in[0,1):\left|x-\frac{p}{q}\right|<\frac{1}{q^{1+\tau}}\right\}.
\end{align}
$W(\tau)$ coincides precisely with $W(\psi)$ when $\psi$ is chosen as $\psi(q)\coloneqq1/q^\tau$ for any $q\in\mathbb{N}$. The Jarn\'ik--Besicovitch Theorem, established independently by Jarn\'ik and Besicovitch, is fundamental result concerning the Hausdorff dimension of $W(\tau)$.
\begin{thm*}[Jarn\'ik \cite{Jarnik1929metrischen}, 1928; Besicovitch \cite{besicovitch1934sets}, 1934]\label{thm: DL Jarn\'ik-Besicovitch}
For any $\tau\geq1$,
\begin{align*}
    \dim{W(\tau)}=\frac{2}{1+\tau}.
\end{align*}
\end{thm*}
The Jarn\'ik--Besicovitch Theorem is generalised by the result of Dodson \cite[Theorem 2]{Dodson1992}, extending the theorem from $W(\tau)$ to $W(\psi)$ for general non-increasing $\psi$. The Hausdorff dimension of $W(\psi)$ is expressed in terms of the lower order at infinity of $1/\psi$, an asymptotic behaviour of $\psi$.
\begin{thm*}[Dodson {\cite[Theorem 2]{Dodson1992}}, 1992]
Let $\psi:\mathbb{N}\to[0,1/2]$. Suppose $\psi$ is non-increasing. Then
\begin{align*}
    \dim{W(\psi)}=\frac{2}{1+\underline{\tau}_{\psi}},
\end{align*}
where $\underline{\tau}_{\psi}\coloneqq\liminf_{q\to+\infty}-\log{\psi(q)}/\log{q}\geq1$.
\end{thm*}

In summary, the results of Khinchine, Jarn\'ik--Besicovitch, and Dodson focus on the metric properties of well approximable numbers, specifically their Lebesgue measure and Hausdorff dimension. It is worth to notice that the fractions $p/q$ that approximate $x\in[0,1)$ in \eqref{eq: Khinchine ineq} are not required to be in their simplest forms nor specified forms. However, imposing additional constraints on the approximating fractions may lead to a reduction in both measure and dimension. This paper presents analogous results under an alternative setting, where the approximating fractions $p/q$ in \eqref{eq: Khinchine ineq} are required to be a L\"uroth convergent of $x$.

\section{Preliminary}

For any $x\in(0,1]$, there exists a unique sequence $(d_n)_{n\in\mathbb{N}}$ of positive integers, each greater than 1, referred to as digits, such that
\begin{align}
    x
    &=[d_1,d_2,d_3,\ldots,d_n,\ldots] \label{eq: Luroth Repr}\\
    &\coloneqq\sum_{k=1}^{\infty}\frac{1}{d_k\prod_{j=1}^{k-1}d_j(d_j-1)}; \label{eq: Luroth Expa}
\end{align}
where for any $n\in\mathbb{N}$, $d_n=d_n(x)\in\mathbb{N}\setminus\{1\}$; \eqref{eq: Luroth Repr} is referred to as the L\"uroth representation of $x$, and \eqref{eq: Luroth Expa} as the L\"uroth expansion of $x$. The sequence $(d_n)_{n\in\mathbb{N}}$ can be determined by an iterative process induced by the L\"uroth map $T:(0,1]\to(0,1]$, defined by for any $x\in(0,1]$,
\begin{align*}
    T(x)\coloneqq\left\lfloor\frac{1}{x}\right\rfloor\left(\left\lfloor\frac{1}{x}+1\right\rfloor x-1\right),
\end{align*}
As illustrated in Figure~\ref{fig: Luroth map}, the L\"uroth map consists of countably many linear pieces, each defined by a distinct linear equation. For any $x\in(0,1]$, the first digit of $x$ in the L\"uroth representation is given by $d_1(x)=\lfloor1/x\rfloor+1$, and for any $n\in\mathbb{N}\setminus\{1\}$, $d_n(x)=d_1(T^{n-1}(x))$. 

\begin{figure}[H]
\centering
\begin{tikzpicture}
\begin{axis}[
    axis lines = left,
    axis x line=center,
    axis y line=center,
    width=0.618\textwidth,
    height=0.618\textwidth,
    xmin=-.1,
    xmax=1.15,
    ymin=-.1,
    ymax=1.15,
    xtick={1,1/2,1/3},
    xticklabels={1,$1/2$,$1/3$},
    ytick={1},
    xlabel = \(x\),
    ylabel = \(T(x)\),
    axis equal,
]
\foreach \d in {2, 3, ..., 13} {
    \addplot[domain={1/\d:1/(\d-1)}, samples=2] {(\d-1)*(\d*x-1)};
}
\node at (1/26,0.5) {$\cdots$};
\end{axis}
\end{tikzpicture}
\caption{L\"uroth Map $T:(0,1]\to(0,1]$}
\label{fig: Luroth map}
\end{figure}

Let $x\in(0,1]$ and $n\in\mathbb{N}$. Define $x_n$, the $n$-th L\"uroth convergent of $x$, as the $n$-th partial sum in \eqref{eq: Luroth Expa}; that is 
\begin{align*}
    x_n
    &\coloneqq[d_1,d_2,d_3,\ldots,d_n] \coloneqq\sum_{k=1}^n\frac{1}{d_k\prod_{j=1}^{k-1}d_j(d_j-1)} \\
    &=\frac{1}{d_1}+\frac{1}{d_1(d_1-1)d_2}+\frac{1}{d_1(d_1-1)d_2(d_2-1)d_3}+\cdots+\frac{1}{d_1(d_1-1)\cdots d_{n-1}(d_{n-1}-1)d_n}.
\end{align*}
The unsimplified numerator $P_n(x)$ and denominator $Q_n(x)$ of the $n$-th L\"uroth convergent of $x$ are respectively defined  by,
\begin{align}
    \label{eq: P_n(x)}
    P_n(x) &\coloneqq [d_1,d_2,\ldots,d_n]Q_n(x), \\
    \label{eq: Q_n(x)}
    Q_n(x) &\coloneqq d_n\prod_{j=1}^{n-1}d_j(d_j-1)=d_1(d_1-1)d_2(d_2-1)\cdots d_n,
\end{align}
with the convention that the empty product equals 1. Note that the fraction $P_n(x)/Q_n(x)$ may not be in its simplest form. For example, for $x=27/71=[3,4,3,4,3,4,\ldots]$ and any $n\in\mathbb{N}\setminus\{1\}$, $\gcd{(P_n(x),Q_n(x))}=2>1$.

Let $\psi:\mathbb{N}\to(0,1]$. Define $L(\psi)$, the set of all L\"uroth $\psi$-well approximable numbers, as
\begin{align*}
    L(\psi)
    \coloneqq
    \limsup_{n\to+\infty}
    \left\{x\in(0,1]:\left|x-\frac{P_n(x)}{Q_n(x)}\right|<\frac{\psi(Q_n(x))}{Q_n(x)}\right\}.
\end{align*}
In the setting of $L(\psi)$, the approximating fractions $p/q$ in \eqref{eq: Khinchine ineq} are required to be L\"uroth convergents of $x$. Let $\tau\geq0$. Define $L(\tau)$, the set of all L\"uroth $\tau$-well approximable numbers, as
\begin{align}
    \label{eq: L tau}
    L(\tau)
    \coloneqq\limsup_{n\to+\infty}
    \left\{x\in(0,1]:\left|x-\frac{P_n(x)}{Q_n(x)}\right|<\frac{1}{{Q_n(x)}^{1+\tau}}\right\}.
\end{align}
$L(\tau)$ coincides precisely with $L(\psi)$ when $\psi$ is chosen as $\psi(q)\coloneqq1/q^\tau$ for any $q\in\mathbb{N}$. The above definitions can be compared with $W(\psi)$ and $W(\tau)$ defined in \eqref{eq: W psi} and \eqref{eq: W tau} respectively.

\section{Main Results}

Theorem~\ref{thm: DL 1} is an analogue of the Khintchine Theorem, which is claimed in the previous research of Tan--Zhou \cite{tan2021approximation}. Theorem~\ref{thm: DL 2} is an analogue of the Jarn\'ik--Besicovitch Theorem, which improves the previous result of Cao--Wu--Zhang \cite[Theorem 1.2]{cao2013efficiency}. Theorems~\ref{thm: DL 3} and \ref{thm: DL 4} are generalisations of Theorem~\ref{thm: DL 2} for general functions. Theorem~\ref{thm: DL 5} is a counterexample to a conjecture stated by Tan--Zhou \cite[Conjecture 2]{tan2021approximation}, an analogue of the result of Dodson. Theorem~\ref{thm: DL 6} is a partial result to a revised conjecture.

Theorem~\ref{thm: DL 1} is originally claimed by Tan--Zhou \cite{tan2021approximation}, and can be seen as an analogue of the Khintchine Theorem.
\begin{thm}[Tan--Zhou \cite{tan2021approximation,tan2021dimension}, 2021]\label{thm: DL 1}
Let $\psi:\mathbb{N}\to(0,1]$. Suppose $\psi$ is non-increasing. Then
\begin{align*}
    \mathcal{L}\left(L(\psi)\right)=
    \begin{dcases}
    0,& \text{if}\quad\sum\nolimits_{q=1}^{\infty}\frac{-\psi(q)\log{\psi(q)}}{q}<+\infty;\\
    1,& \text{if}\quad\sum\nolimits_{q=1}^{\infty}\frac{-\psi(q)\log{\psi(q)}}{q}=+\infty.
    \end{dcases}
\end{align*}
\end{thm}
Let $\psi:\mathbb{N}\to(0,1/2]$ be non-increasing. The divergence of $\sum\nolimits_{q=1}^{\infty}-\psi(q)\log{\psi(q)}/q$ implies the divergence of $\sum\nolimits_{q=1}^{\infty}\psi(q)$; however, the converse does not necessarily hold. Thus, $\mathcal{L}(L(\psi))\leq\mathcal{L}(W(\psi))$ is obtained, which aligns with the intuitive expectation that imposing additional restrictions on the approximating fractions can only reduce the Lebesgue measure.

The statement of Theorem~\ref{thm: DL 1} appears in \cite{tan2021approximation}, and is claimed to be studied in \cite{tan2021dimension}. However, only an equivalent statement of the theorem is presented in \cite{tan2021dimension}, without explicitly stating the equivalence or providing its proof. In this paper, a supplementary proof of Theorem~\ref{thm: DL 1} is given.

The following dimensional result is proved by Cao--Wu--Zhang \cite[Theorem 1.2]{cao2013efficiency}, and can be seen as an analogue of the Jarn\'ik--Besicovitch Theorem.
\begin{thm*}[Cao--Wu--Zhang {\cite[Theorem 1.2]{cao2013efficiency}}, 2013]
For any $\tau\geq1$, 
\begin{align*}
    \dim{L(\tau)}=\frac{1}{1+\tau}.
\end{align*}
\end{thm*}
By comparing with the Jarn\'ik--Besicovitch Theorem, the result above states that for any $\tau\geq1$, the Hausdorff dimension of $L(\tau)$ is exactly half that of $W(\tau)$. This comparison again aligns with the intuitive expectation that imposing additional restrictions can only reduce the Hausdorff dimension.

In some works \cite{tan2021approximation,tan2021dimension}, the result of Cao--Wu--Zhang is claimed true for all non-negative $\tau\geq0$, though no explicit explanation is provided. In this paper, Theorem~\ref{thm: DL 2} is proved to verify the claim with extra information on its Hausdorff measure. 
\begin{thm}\label{thm: DL 2}
For any $\tau\geq0$, 
\begin{align*}
    \dim{L(\tau)}=\frac{1}{1+\tau},
\end{align*}
and the $1/(1+\tau)$-Hausdorff measure of $L(\tau)$ is infinite, that is
\begin{align*}
    \mathcal{H}^{1/(1+\tau)}{({L(\tau)})}=+\infty.
\end{align*}
\end{thm}
In the proof, the Beresnevich--Velani Mass Transference Principle \cite[Theorem 2]{beresnevich2006mass} is applied. This approach not only extends the range from the previous result $\tau\geq1$ to $\tau\geq0$, but also simplifies the argument in \cite{cao2013efficiency}, and provides the Hausdorff measure at the critical exponent.

Before discussing the analogue of the result of Dodson in the L\"uroth setting, a few generalisations to Theorem~\ref{thm: DL 2} are given. Let $\psi:\mathbb{N}\to(0,1]$. Define the lower and upper orders at infinity of $1/\psi$ respectively by
\begin{align}
\label{eq: tau}
    \underline{\tau}_\psi
    \coloneqq\liminf_{q\to+\infty}-\frac{\log{\psi(q)}}{\log{q}};
    &&
    \overline{\tau}_\psi
    \coloneqq\limsup_{q\to+\infty}-\frac{\log{\psi(q)}}{\log{q}}.
\end{align}
\begin{thm}\label{thm: DL 3}
Let $\psi:\mathbb{N}\to(0,1]$. 
\begin{align*}
    \frac{1}{1+\overline{\tau}_\psi}\leq\dim{L(\psi)}\leq\frac{1}{1+\underline{\tau}_\psi}.
\end{align*}
\end{thm}
Theorem~\ref{thm: DL 3} provides both lower and upper bounds on the Hausdorff dimension of $L(\psi)$, and serves as a generalisation of Theorem~\ref{thm: DL 2}. For $\tau\geq0$ and $\psi:\mathbb{N}\to(0,1]$ taken as $\psi(q)\coloneqq 1/q^{\tau}$ for any $q\in\mathbb{N}$, Theorem~\ref{thm: DL 2} is recovered as $\underline{\tau}_\psi=\overline{\tau}_\psi=\tau$.

From its definition, $L(\psi)$ depends only on the values of $\psi$ evaluated at a specific subset of $\mathbb{N}$ with arbitrarily small density. Since for any $x\in(0,1]$ and $n\in\mathbb{N}$, the denominator $Q_n(x)$ of the $n$-th L\"uroth convergent of \(x\) defined in \eqref{eq: Q_n(x)}, is divisible by $2^{n-1}$. Thus, the contribution of $\psi$ evaluated at non-highly composite numbers is negligible in determining $L(\psi)$, and hence its measure and dimension. Theorem~\ref{thm: DL 4} refines Theorem~\ref{thm: DL 3} by incorporating a precise range for the limit inferior and limit superior in \eqref{eq: tau}.
\begin{thm}\label{thm: DL 4}
Let $\psi:\mathbb{N}\to(0,1]$. 
\begin{align*}
    \frac{1}{1+\overline{\lambda}_\psi}
    \leq\dim{L(\psi)}
    \leq\frac{1}{1+\underline{\lambda}_\psi},
\end{align*}
where $\underline{\lambda}_\psi$ and $\overline{\lambda}_\psi$ are respectively defined by,
\begin{align*}
    \underline{\lambda}_\psi
    \coloneqq
    \liminf_{k\to+\infty}\liminf_{q\in S_k}\frac{-\log{\psi(q)}}{\log{q}}
    \geq\underline{\tau}_{\psi}, &&
    \overline{\lambda}_\psi
    \coloneqq
    \liminf_{k\to+\infty}\limsup_{q\in S_k}\frac{-\log{\psi(q)}}{\log{q}}
    \leq\overline{\tau}_{\psi},
\end{align*}
where for any $k\in\mathbb{N}$,
\begin{align*}
    S_k
    \coloneqq\left\{d_k\prod\nolimits_{j=1}^{k-1}d_j(d_j-1):\text{for any $n\in\mathbb{N}$, }d_n\in\mathbb{N}\setminus\{1\}\right\}\subset\mathbb{N}.
\end{align*}
\end{thm}

The following conjecture is stated by Tan--Zhou \cite{tan2021approximation}, regarding the Hausdorff dimension of $L(\psi)$, and can be seen as an analogue of the result of Dodson.
\begin{conj*}[Tan--Zhou {\cite[Conjecture 2]{tan2021approximation}}, 2021]
Let $\psi:\mathbb{N}\to(0,1]$.
\begin{align*}
    \dim{L(\psi)}=\frac{1}{1+\underline{\tau}_\psi}.
\end{align*}
\end{conj*}

Theorem~\ref{thm: DL 5} demonstrates that the conjecture can fail without the monotonicity assumption. Consequently, the conjecture should be reformulated to include that assumption.
\begin{thm}\label{thm: DL 5}
For any $\tau\geq0$, there exists $\psi:\mathbb{N}\to(0,1]$ such that $\psi$ is not eventually non-increasing, 
$\underline{\tau}_\psi=\tau$, and 
\begin{align*}
    \dim{L(\psi)}=0<\frac{1}{1+\underline{\tau}_\psi}.
\end{align*}
\end{thm}
\begin{conj*}
Let $\psi:\mathbb{N}\to(0,1]$. Suppose $\psi$ is non-increasing. Then,
\begin{align*}
    \dim{L(\psi)}=\frac{1}{1+\underline{\tau}_\psi}.
\end{align*}
\end{conj*}
Theorem~\ref{thm: DL 6} provides a partial result to the revised conjecture with three extra assumptions: a sufficiently large lower order, a stronger monotonicity condition, and the divergence of a specific series. Let $\psi:\mathbb{N}\to(0,1]$ and define $\theta_s:\mathbb{N}\to\mathbb{R}^+$ by, for any $q\in\mathbb{N}$,
\begin{align}
    \label{eq: theta s}
    \theta_s(q)\coloneqq q^{1-s}\psi^s(q),
\end{align}
where $s\coloneqq 1/(1+\underline{\tau}_{\psi})$. Note that if $\theta_s$ is non-increasing then $\psi$ is also non-increasing.
\begin{thm}\label{thm: DL 6}
    Let $\psi:\mathbb{N}\to(0,1]$. Suppose, $\underline{\tau}_{\psi}>1$, $\theta_s$ is non-increasing, and 
    \begin{align*}
        \sum_{q=1}^\infty\left(\frac{\psi(q)}{q}\right)^s\log{q}=+\infty,
    \end{align*}
    where $\underline{\tau}_\psi$ and $\theta_s$ are defined in \eqref{eq: tau} and \eqref{eq: theta s} respectively. Then,
    \begin{align*}
        \dim{L(\psi)}=\frac{1}{1+\underline{\tau}_{\psi}}.
    \end{align*}
\end{thm}

\section{Proofs}

\subsection{Proof of Theorem~\ref{thm: DL 1}}\label{Section: Luroth Lebesgue}

Theorem~\ref{thm: DL 1} presents a measure-theoretic result originally stated by Tan--Zhou \cite{tan2021approximation}, which forms an analogue of the Khintchine Theorem in the L\"uroth setting. While the original statement is claimed in \cite{tan2021approximation}, a direct proof is not fully provided.

Theorem~\ref{thm: DL 1} is logically equivalent to \cite[Corollary 1.3]{tan2021dimension}. The following demonstrates how the former is implied by the latter.

\begin{proof}[Proof of Theorem~\ref{thm: DL 1}]
By \cite[(4.1)]{tan2021dimension}, for any $x\in(0,1]$ and $n\in\mathbb{N}$,
\begin{align*}
    \frac{1}{(d_n(x)-1)d_{n+1}(x)}<\left|xQ_n(x)-P_n(x)\right|<\frac{1}{(d_n(x)-1)(d_{n+1}(x)-1)}.
\end{align*}
Since $d_n(x)\geq2$ for all $x\in(0,1]$ and $n\in\mathbb{N}$, the above implies that
\begin{align*}
    \frac{1}{d_n(x)d_{n+1}(x)}
    <|xQ_n(x)-P_n(x)|
    <\frac{4}{d_n(x)d_{n+1}(x)}.
\end{align*}
The following chain of set inclusions holds:
\begin{align*}
    &\limsup_{n\to\infty}\left\{x\in(0,1]:d_n(x)d_{n+1}(x)\geq\frac{4}{\psi(Q_n(x))}\right\} \\
    &\subset\limsup_{n\to\infty}\left\{x\in(0,1]:\left|xQ_n(x)-P_n(x)\right|<\psi(Q_n(x))\right\} = L(\psi) \\
    &\subset
    \limsup_{n\to\infty}\left\{x\in(0,1]:d_n(x)d_{n+1}(x)\geq\frac{1}{\psi(Q_n(x))}\right\}.
\end{align*}

Suppose, in the convergent case,
\begin{align*}
    \sum_{q=1}^{\infty}-\frac{\psi(q)\log{\psi(q)}}{q}<+\infty.
\end{align*}
In particular, $\lim_{q\to\infty}\psi(q)=0$. Define a non-decreasing $\varphi_0:\mathbb{N}\to[2,+\infty)$ by, for any $q\in\mathbb{N}$,
\begin{align*}
    \varphi_0(q)\coloneqq\max{\left(2,\frac{1}{\psi(q)}\right)}.
\end{align*}
Since $\lim_{q\to\infty}\psi(q)=0$, there exists a minimal $q_0\in\mathbb{N}$ such that for any $q\in\mathbb{N}$, if $q>q_0$ then $\varphi_0(q)={1}/{\psi(q)}$. By straightforward computation,
\begin{align*}
    \sum_{q=1}^\infty\frac{\log{\varphi_0(q)}}{q\varphi_0(q)}
    &=\sum_{q=1}^{q_0}\frac{\log{2}}{2q}
    +\sum_{q=q_0+1}^{\infty}-\frac{\psi(q)\log{\psi(q)}}{q} \\
    &\leq\sum_{q=1}^{q_0}\frac{\log{2}}{2q}
    +\sum_{q=1}^{\infty}-\frac{\psi(q)\log{\psi(q)}}{q}
    <+\infty.
\end{align*}
By the convergent part of \cite[Corollary 1.3]{tan2021dimension}, the set inclusion implies that $\mathcal{L}(L(\psi))=0$.

Suppose, in the divergent case,
\begin{align*}
    \sum_{q=1}^{\infty}-\frac{\psi(q)\log{\psi(q)}}{q}=+\infty.
\end{align*}
Define a non-decreasing $\varphi_1:\mathbb{N}\to[2,+\infty)$ by, for any $q\in\mathbb{N}$,
\begin{align*}
    \varphi_1(q)\coloneqq \frac{4}{\psi(q)}.
\end{align*}
By a straightforward computation,
\begin{align*}
    \sum_{q=1}^\infty\frac{\log{\varphi_1(q)}}{q\varphi_1(q)}
    &=\sum_{q=1}^{\infty}\frac{-\psi(q)\log{(\psi(q)/4)}}{4q} \\
    &\geq\sum_{q=1}^{\infty}\frac{-\psi(q)\log{\psi(q)}}{q}
    =+\infty.
\end{align*}
By the divergent part of \cite[Corollary 1.3]{tan2021dimension}, the set inclusion implies that $\mathcal{L}(L(\psi))=1$.

The proof of Theorem~\ref{thm: DL 1} is completed.
\end{proof}

\subsection{Proof of Theorem~\ref{thm: DL 2}}

Theorem~\ref{thm: DL 2} establishes the Hausdorff dimension of the set of L\"uroth \( \tau \)-well-approximable numbers, improving the result of Cao--Wu--Zhang, which forms an analogue of the Jarn\'ik--Besicovitch Theorem in the L\"uroth setting.

Proposition~\ref{prop: Luroth-Dirichlet} provides an inequality that every L\"uroth convergent satisfies.
\begin{prop}\label{prop: Luroth-Dirichlet}
For any $x\in(0,1]$ and $n\in\mathbb{N}$,
\begin{align*}
    0< x-\frac{P_n(x)}{Q_n(x)}\leq\frac{1}{d_n(x)-1}\frac{1}{Q_n(x)}.
\end{align*}
\end{prop}
\begin{proof}
Pick any $x\in(0,1]$ and $n\in\mathbb{N}$. By the definitions of $P_n(x)$ and $Q_n(x)$ in \eqref{eq: P_n(x)} and \eqref{eq: Q_n(x)},
\begin{align*}
    x-\frac{P_n(x)}{Q_n(x)}
    =\sum_{k={n+1}}^\infty\frac{1}{d_k\prod_{j=1}^{k-1}d_j(d_j-1)}
    = \frac{1}{d_n-1}\frac{1}{Q_n(x)}\sum_{k={n+1}}^\infty\frac{1}{d_k\prod_{j=n+1}^{k-1}d_j(d_j-1)},
\end{align*}
where the value of the summation in the rightmost expression lies in the interval $(0,1]$.
\end{proof}
Proposition~\ref{prop: Luroth-Dirichlet} states that every real number in $(0,1]$ is strictly greater than all of its L\"uroth convergents; in other words, all L\"uroth convergents approximate the real number from the left of the number line. Figure~\ref{fig: from the left} is an illustration.
\newlength{\mytikzwidth}
\setlength{\mytikzwidth}{0.95\textwidth}
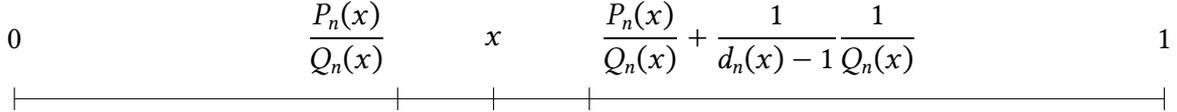
\begin{figure}[H]
\begin{center}
\begin{tikzpicture}
    \draw (0,0) -- (\mytikzwidth,0);
    \draw (0.000\mytikzwidth,-0.15) -- + (0,0.3); 
    \draw (1.000\mytikzwidth,-0.15) -- + (0,0.3);
    \draw (.5000\mytikzwidth,-0.15) -- + (0,0.3);
    \draw (.4167\mytikzwidth,-0.15) -- + (0,0.3);
    \draw (.3333\mytikzwidth,-0.15) -- + (0,0.3);
    \node at (0.0000\mytikzwidth,.8) {$0$};
    \node at (1.0000\mytikzwidth,.8) {$1$};
    \node at (0.5000\mytikzwidth,.8) [right] {$\displaystyle \frac{P_n(x)}{Q_n(x)}+\frac{1}{d_n(x)-1}\frac{1}{Q_n(x)}$};
    \node at (0.4167\mytikzwidth,.8) {$x$};
    \node at (0.3333\mytikzwidth,.8) [left] {$\displaystyle \frac{P_n(x)}{Q_n(x)}$};
\end{tikzpicture}
\end{center}
\caption{L\"uroth Convergents Approximate from the Left}
\label{fig: from the left}
\end{figure}

By Proposition~\ref{prop: Luroth-Dirichlet}, for any $\tau\geq0$, the set $L(\tau)$, defined in \eqref{eq: L tau}, can be expressed as:
\begin{align}
    \label{eq: L tau better}
    L(\tau)
    =\limsup_{n\to+\infty}
    \left\{x\in(0,1]:0<x-\frac{P_n(x)}{Q_n(x)}<\frac{1}{{Q_n(x)}^{1+\tau}}\right\}.
\end{align}
The proof of the upper bound for $\dim{L(\tau)}$ is established by a standard covering argument in fractal geometry.

\begin{prop}[Upper Bound of $\dim{L(\tau)}$] \label{prop: dim L tau upper}
For any $\tau\geq0$,
\begin{align*}
    \dim{L(\tau)}\leq\frac{1}{1+\tau}.
\end{align*}
\end{prop}
\begin{proof}
Pick any $n\in\mathbb{N}$ and $d_1,d_2,\ldots,d_n\in\mathbb{N}\setminus\{1\}$. Define $Q_n\coloneqq d_1(d_1-1)\dots d_{n-1}(d_{n-1}-1)d_n\geq2^n$, $P_n\coloneqq [d_1,d_2,\ldots,d_n]Q_n$, and
\begin{align*}
    I_{n,\tau}(d_1,d_2,\ldots,d_n)
    \coloneqq
    \left(\frac{P_n}{Q_n},\frac{P_n}{Q_n}+\frac{1}{{Q_n}^{1+\tau}}\right).
\end{align*}
Note that the diameter of ${I_{n,\tau}}$ satisfies that $\operatorname{diam}{I_{n,\tau}}=1/{Q_n}^{1+\tau}\leq1/2^n$. Pick any $\rho>0$ and $N\in\mathbb{N}$. Suppose $N>-\log{\rho}/\log{2}$. By \eqref{eq: L tau better}, the following is a $\rho$-cover of $L(\tau)$.
\begin{align}
    \label{eq: rho cover DL 2}
    \bigcup_{n=N}^{\infty}\bigcup_{(d_1,\ldots,d_n)\in(\mathbb{N}\setminus\{1\})^n}
    I_{n,\tau}(d_1,d_2,\ldots,d_n).
\end{align}
Pick any $\varepsilon>0$ and $s>1/(1+\tau)$. Without loss of generality, $N\in\mathbb{N}$ is large enough in terms of $\tau\geq0$, $s>0$ and $\varepsilon>0$, so that
\begin{align*}
    \sum_{n=N}^{\infty}{r_{\tau,s}}^{n-1}<\frac{\varepsilon}{\zeta((1+\tau)s)}
    ,
\end{align*}
where $\zeta$ is the Riemann zeta function and
\begin{align*}
    r_{\tau,s}
    \coloneqq\sum_{d=2}^\infty\frac{1}{{\left(d(d-1)\right)}^{(1+\tau)s}}
    \in(0,1)
    .
\end{align*}
Hence, the $s$-Hausdorff measure for the $\rho$-cover \eqref{eq: rho cover DL 2} is obtained as
\begin{align*}
    \mathcal{H}^s_{\rho}(L(\tau))
    &\leq\sum_{n=N}^{\infty}\sum_{(d_1,\ldots,d_n)\in(\mathbb{N}\setminus\{1\})^n}\left(\operatorname{diam}I_{n,\tau}(d_1,\ldots,d_n)\right)^s \\
    &=\sum_{n=N}^{\infty}\sum_{(d_1,\ldots,d_n)\in(\mathbb{N}\setminus\{1\})^n}\frac{1}{{Q_{n}(d_1,\ldots,d_n)}^{(1+\tau)s}} \\
    &=\sum_{n=N}^{\infty}\sum_{(d_1,\ldots,d_n)\in(\mathbb{N}\setminus\{1\})^n}\frac{1}{{\left(d_1(d_1-1)\cdots d_{n-1}(d_{n-1}-1)d_n\right)}^{(1+\tau)s}} \\
    &=\sum_{n=N}^{\infty}\left(\sum_{d=2}^\infty\frac{1}{{\left(d(d-1)\right)}^{(1+\tau)s}}\right)^{n-1}\left(\sum_{d=2}^\infty\frac{1}{{d}^{(1+\tau)s}}\right) \\
    &<\zeta((1+\tau)s)\sum_{n=N}^{\infty}{r_{\tau,s}}^{n-1}
    <\varepsilon
    .
\end{align*}
Thus, for any $s>1/(1+\tau)$, $\mathcal{H}^s(L(\tau))=0$. The result follows from the definition of Hausdorff dimension.
\end{proof}

The proof of the lower bound for $\dim{L(\tau)}$ follows from applying the Beresnevich--Velani Mass Transference Principle \cite[Theorem 2]{beresnevich2006mass}, providing an alternative to the method presented in \cite{cao2013efficiency}. Proposition~\ref{prop: AMTP} is sufficient for the rest of the paper.
\begin{prop}[Applied Mass Transference Principle]\label{prop: AMTP}
    Let $(I_n)_{n\in\mathbb{N}}$ be a sequence of intervals in $[0,1]$ and $0\leq s<1$. Suppose $\lim_{n\to+\infty}\operatorname{diam}{I_n}=0$ and
    \begin{align*}
        \mathcal{L}\left([0,1]\cap\limsup_{n\to+\infty}{I_n}^s\right)=1,
    \end{align*}
    where $I^s$ denotes the interval with the same centre as $I$ and with radius equal to the radius of $I$ raised to the power of $s$.
    Then
    \begin{align*}
        \mathcal{H}^s\left(\limsup_{n\to+\infty}{I_n}\right)=+\infty;
    \end{align*}
    in particular, 
    \begin{align*}
        \dim{\left(\limsup_{n\to+\infty}{I_n}\right)}\geq s.
    \end{align*}
\end{prop}
\begin{proof}
    The proof essentially applies \cite[Theorem 2]{beresnevich2006mass}.
    
    Suppose $0<s<1$. Define $f_s:\mathbb{R}^+\to\mathbb{R}^+$ by, for any $x\in\mathbb{R}^+$, $f_s(x)\coloneqq x^s$. $f_s$ is a dimension function, that is, $f_s$ is non-decreasing and $\lim_{x\to0^+}f_s(x)=0$. Define $X_s$ and its 1-periodic extension $Y_s$ by,
    \begin{align*}
        X_s\coloneqq\limsup_{n\to+\infty}\left([0,1)\cap{I_n}^s\right), &&
        Y_s\coloneqq\bigcup_{z\in\mathbb{Z}}\left(X_s+z\right),
    \end{align*}
    where $X_s+z$ is the translation of $X_s$ by $z$, defined as $X_s+z\coloneqq\{x+z:x\in X_s\}$.

    Pick any interval $I\subset\mathbb{R}$. Pick any $z\in\mathbb{Z}$. By assumption, $\mathcal{L}(X_s)=1$, implying $\mathcal{L}([z,z+1)\cap Y_s)=1$ and $\mathcal{L}([z,z+1)\setminus Y_s)=0$. By the additivity of Lebesgue measure,
    \begin{align*}
        \mathcal{L}\left(I\cap[z,z+1)\cap Y_s\right)
        =\mathcal{L}(I\cap[z,z+1))-\mathcal{L}(I\cap[z,z+1)\setminus Y_s)=\mathcal{L}(I\cap[z,z+1)).
    \end{align*}
    By the countable additivity of Lebesgue measure, the above implies that $\mathcal{L}(I\cap Y_s)=\mathcal{L}(I)$. By a re-enumeration, $Y_s=\limsup_{n\to\infty}{B_n}^s$, where for any $n\in\mathbb{N}$, $B_n$ is an interval, and $\lim_{n\to+\infty}\operatorname{diam}{B_n}=0$. Thus, the assumption of \cite[Theorem 2]{beresnevich2006mass} is satisfied and its conclusion implies that,
   \begin{align*}
        \mathcal{H}^s\left([0,1]\cap Y_s\right)=\mathcal{H}^s([0,1])=+\infty.
    \end{align*}
    By $X_s=[0,1)\cap Y_s$, the above implies that
    \begin{align*}
        \mathcal{H}^s\left(\limsup_{n\to+\infty}{I_n}\right)\geq\mathcal{H}^s\left(X_s\right)=+\infty.
    \end{align*}
    The result follows from the definition of Hausdorff dimension.
\end{proof}

In order to apply Proposition~\ref{prop: AMTP} to establish a lower bound for $\dim{L(\tau)}$, $L(\tau)$ should be expressed as a $\limsup$ set of intervals. Define $S$ to be the set of all 3-tuples consisting of the L\"uroth numerators, denominators, and last digits. Formally,
\begin{align}
\label{eq: def S}
    S\coloneqq\bigcup_{n\in\mathbb{N}}\bigcup_{x\in(0,1]}\left(P_n(x),Q_n(x),d_n(x)\right)
    ,
\end{align}
where for any $n\in\mathbb{N}$ and $x\in(0,1]$, $P_n(x)$ and $Q_n(x)$ are defined in \eqref{eq: P_n(x)} and \eqref{eq: Q_n(x)} respectively, and $d_n(x)=d_n$ is given in \eqref{eq: Luroth Repr}.
Note that $S$ is countable. Let
\begin{align}
\label{eq: S re-enumeration}
    S=((P_k,Q_k,d_k))_{k\in\mathbb{N}}
\end{align}
be a re-enumeration so that $(Q_k)_{k\in\mathbb{N}}$ is non-decreasing. Proposition~\ref{Luroth, supset} can be compared with Proposition~\ref{prop: Luroth-Dirichlet}, stating that satisfying an inequality implies that the fraction is a L\"uroth convergent.
\begin{prop}\label{Luroth, supset}
For any $x\in(0,1]$ and $k\in\mathbb{N}$, if
\begin{align*}
    0<x-\frac{P_k}{Q_k}\leq\frac{1}{(d_k-1)Q_k},
\end{align*}
then $P_k/Q_k$ is a L\"uroth convergent of $x$.
\end{prop}
\begin{proof}
By the definition of $S$, there exist $y\in(0,1]$ and $n\in\mathbb{N}$ such that $Q_k=Q_n(y)$. By the definition, there exist $d_1,d_2,\ldots,d_n\in\mathbb{N}\setminus\{1\}$ such that 
\begin{align*}
    Q_k
    =d_n\prod_{i=1}^{n-1}d_i(d_i-1)
    =d_1(d_1-1)d_2(d_2-1)\cdots d_{n-1}(d_{n-1}-1)d_n.
\end{align*}
The length of the cylinder:
\begin{align*}
    \left\{z\in(0,1]:d_1(z)=d_1,d_2(z)=d_2,\ldots,d_n(z)=d_n\right\}.
\end{align*}
for $[d_1,d_2,\ldots,d_n]$ is given by
\begin{align*}
    \frac{1}{\prod_{i=1}^{n}d_i(d_i-1)}
    =\frac{1}{(d_n-1)Q}.
\end{align*}
Therefore, $P_k/Q_k$ is the $n$-th L\"uroth convergent of $x$.
\end{proof}
Define, for any $\tau\geq0$, a sequence of intervals $(B_{\tau,k})_{k\in\mathbb{N}}$ in $[0,1]$ by, for any $k\in\mathbb{N}$, 
\begin{align}
    \label{eq: B tau k}
    B_{\tau,k}
    \coloneqq\left(\frac{P_k}{Q_k},\frac{P_k}{Q_k}+\frac{1}{{Q_k}^{1+\tau}}\right)\cap\left(\frac{P_k}{Q_k},\frac{P_k}{Q_k}+\frac{1}{(d_k-1)Q_k}\right].
\end{align}
\begin{prop}\label{prop: Luroth, limsup}
For any $\tau\geq0$,
\begin{align}\label{eq: Luroth, limsup}
    L(\tau)
    &=\limsup_{k\to+\infty}B_{\tau,k}.
\end{align}
\end{prop}
\begin{proof}
    The inclusion $\subset$ follows directly from \eqref{eq: L tau better} with the definitions of $S$ and $B_{\tau,k}$ in \eqref{eq: def S} and \eqref{eq: B tau k} respectively. The reverse inclusion $\supset$ follows from Proposition~\ref{Luroth, supset}.
\end{proof}

\begin{prop}[Lower Bound of $\dim{L(\tau)}$]\label{prop: dim L tau lower}
For any $\tau\geq0$,
\begin{align*}
    \dim{L(\tau)}\geq\frac{1}{1+\tau}.
\end{align*}
\end{prop}
\begin{proof}
Pick any $x\in(0,1]\setminus\mathbb{Q}$. Since $x$ is irrational, there exist infinitely many $n\in\mathbb{N}$ such that $d_n(x)\geq3$. By Proposition~\ref{prop: Luroth-Dirichlet}, for any $n\in\mathbb{N}$, if $d_n(x)\geq3$ then
\begin{align}
    \label{eq: Luroth-Dirichlet irrational}
    0< x-\frac{P_n(x)}{Q_n(x)}<\frac{1}{Q_n(x)}.
\end{align}
By \eqref{eq: L tau better}, $(0,1]\setminus\mathbb{Q}\subset L(0)$ and consequently $\dim{L(0)}=1$.

Suppose $\tau>0$. By a straightforward computation, for any $k\in\mathbb{N}$,
\begin{align*}
    \left(\frac{P_k}{Q_k},\frac{P_k}{Q_k}+\frac{1}{Q_k}\right)
    &\subset{B_{\tau,k}}^{1/(1+\tau)}
    ,
\end{align*}
where $B^s$ denotes the interval with the same centre as $B$ and with radius equal to the radius of $B$ raised to the power of $s$. Since for any $x\in(0,1]\setminus\mathbb{Q}$, there exist infinitely many $n\in\mathbb{N}$ such that \eqref{eq: Luroth-Dirichlet irrational} holds, implying that, 
\begin{align*}
    (0,1]\setminus\mathbb{Q}
    &\subset\left((0,1]\setminus\mathbb{Q}\right)\cap\limsup_{k\to+\infty}\left(\frac{P_k}{Q_k},\frac{P_k}{Q_k}+\frac{1}{Q_k}\right) \\
    &\subset\left((0,1]\setminus\mathbb{Q}\right)\cap\limsup_{k\to+\infty}{B_{\tau,k}}^{1/(1+\tau)}
    .
\end{align*}

By the re-enumeration on $S$ in \eqref{eq: S re-enumeration}, $\lim_{k\to+\infty}Q_k=+\infty$ and consequently $\lim_{k\to+\infty}\operatorname{diam}{B_{\tau,k}}=0$. The assumption in Proposition~\ref{prop: AMTP} is satisfied by that,
\begin{align*}
    1
    =\mathcal{L}\left((0,1]\setminus\mathbb{Q}\right)
    \leq
    \mathcal{L}\left((0,1]\setminus\mathbb{Q}\cap\limsup_{k\to+\infty}{B_{\tau,k}}^{1/(1+\tau)}\right)\leq 1.
\end{align*}
By the conclusion in Proposition~\ref{prop: AMTP} and \eqref{eq: Luroth, limsup},
\begin{align}
    \label{eq: H measure of L(tau)}
    \mathcal{H}^{1/(1+\tau)}{(L(\tau))}
    =\mathcal{H}^{1/(1+\tau)}\left(\limsup_{k\to+\infty}{B_{\tau,k}}\right)
    =+\infty;
\end{align}
in particular, 
\begin{align*}
    \dim{L(\tau)}
    =
    \dim{\left(\limsup_{k\to+\infty}B_{\tau,k}\right)}
    \geq
    \frac{1}{1+\tau}.
\end{align*}
Thus, the lower bound of $\dim{L(\tau)}$ is established.
\end{proof}

\begin{proof}[Proof of Theorem~\ref{thm: DL 2}]
By combining Proposition~\ref{prop: dim L tau upper} and \ref{prop: dim L tau lower}, the Hausdorff dimension of $L(\tau)$ is established for all $\tau\geq0$. By \eqref{eq: H measure of L(tau)}, the Hausdorff measure of $L(\tau)$ at the critical exponent is obtained.

The proof of Theorem~\ref{thm: DL 2} is completed.
\end{proof}

\subsection{Proof of Theorem~\ref{thm: DL 3}}

Theorem~\ref{thm: DL 3} seems to be a stronger result than Theorem~\ref{thm: DL 2}, the proof actually demonstrates that the former follows directly by applying latter twice.

\begin{proof}[Proof of Theorem~\ref{thm: DL 3}]
Suppose $\underline{\tau}_\psi>0$. Pick any $0<\varepsilon<\underline{\tau}_\psi$. By the definition of $\underline{\tau}_\psi$, there exists $N_\varepsilon\in\mathbb{N}$ such that for any $q\in\mathbb{N}$, if $q>N_\varepsilon$ then
\begin{align*}
    0<\underline{\tau}_\psi-\varepsilon&<-\frac{\log{\psi(q)}}{\log{q}} \\
    \psi(q)&<\frac{1}{q^{\underline{\tau}_\psi-\varepsilon}}.
\end{align*}
Thus, $L(\psi)\subset L(\underline{\tau}_\psi-\varepsilon)$. By Theorem~\ref{thm: DL 2} and the monotonicity of Hausdorff dimension,
\begin{align*}
    \dim{L(\psi)}
    \leq \dim{L(\underline{\tau}_\psi-\varepsilon)}
    = \frac{1}{1+\underline{\tau}_\psi-\varepsilon}
    .
\end{align*}
Since $\varepsilon>0$ is arbitrary, the upper bound is established.

Pick any $\varepsilon>0$. By the definition of $\overline{\tau}_\psi$, there exists $N_\varepsilon\in\mathbb{N}$ such that for any $q\in\mathbb{N}$, if $q>N_\varepsilon$ then
\begin{align*}
    -\frac{\log{\psi(q)}}{\log{q}}&<\overline{\tau}_\psi+\varepsilon \\
    \frac{1}{q^{\overline{\tau}_\psi+\varepsilon}}&<\psi(q).
\end{align*}
Thus, $L(\overline{\tau}_\psi+\varepsilon)\subset L(\psi)$. By Theorem~\ref{thm: DL 2} and the monotonicity of Hausdorff dimension,
\begin{align*}
    \frac{1}{1+\overline{\tau}_\psi+\varepsilon}
    =\dim{L(\overline{\tau}_\psi+\varepsilon)}
    \leq \dim{L(\psi)}
    .
\end{align*}
Since $\varepsilon>0$ is arbitrary, the lower bound is established.

The proof of Theorem~\ref{thm: DL 3} is completed.
\end{proof}

\subsection{Proof of Theorem~\ref{thm: DL 4}}

Theorem~\ref{thm: DL 4} refines the general dimensional bounds established in Theorem~\ref{thm: DL 3}. The proof for the upper bound follows a similar approach as in Theorem~\ref{thm: DL 3}; and for the lower bound, Proposition~\ref{prop: AMTP} is applied once again.

\begin{proof}[Proof of Theorem~\ref{thm: DL 4}]\label{proof: Luroth Lee dim double bounds better}

Suppose $\underline{\lambda}_\psi>0$. Pick any $0<\varepsilon<\underline{\lambda}_\psi$. By the definition of $\underline{\lambda}_\psi$, there exists $K_\varepsilon\in\mathbb{N}$ such that for any $k\in\mathbb{N}$, there exists $q_k\in\mathbb{N}$ such that for any $q\in S_k$, if $k>K_\varepsilon$ and $q>q_k$ then
\begin{align*}
    \underline{\lambda}_\psi-\varepsilon
    &<-\frac{\log{\psi(q)}}{\log{q}} \\
    \psi(q)
    &<\frac{1}{q^{\underline{\lambda}_\psi-\varepsilon}}.
\end{align*}
Thus, $L(\psi)\subset L(\underline{\lambda}_\psi-\varepsilon)$. By Theorem~\ref{thm: DL 2} and the monotonicity of Hausdorff dimension,
\begin{align*}
    \dim{L(\psi)}
    \leq \dim{L(\underline{\lambda}_\psi-\varepsilon)}
    =\frac{1}{1+\underline{\lambda}_\psi-\varepsilon}
    .
\end{align*}
Since $\varepsilon>0$ is arbitrary, the upper bound is established.

Pick any $\varepsilon>0$. By the definition of $\overline{\lambda}_\psi$, there exists a strictly increasing sequence $(k_j)_{j\in\mathbb{N}}$ of positive integers such that for any $j\in\mathbb{N}$,
\begin{align*}
    \alpha_{j}\coloneqq\limsup_{q\in S_{k_j}}\frac{-\log{\psi(q)}}{\log{q}}<\overline{\lambda}_\psi+\frac{\varepsilon}{2}.
\end{align*}
For any $j\in\mathbb{N}$, by the definition of the limit superior, there exists $q_{k_j}\in S_{k_j}$ such that for any $q\in S_{k_j}$, if $q>q_{k_j}$ then
\begin{align*}
    \overline{\lambda}_\psi+\varepsilon
    >\alpha_{j}+\frac{\varepsilon}{2}
    >\frac{-\log{\psi(q)}}{\log{q}}.
\end{align*}
Thus, for any $\varepsilon>0$, there exists a strictly increasing sequence $(k_j)_{j\in\mathbb{N}}$ of positive integers such that for any $j\in\mathbb{N}$, there exists $q_{k_j}\in S_{k_j}$ such that for any $q\in S_{k_j}$, if $q>q_{k_j}$ then
\begin{align*}
    \frac{-\log{\psi(q)}}{\log{q}}<\overline{\lambda}_\psi+\varepsilon.
\end{align*}
With loss of generality, for any $j\in\mathbb{N}$, $q_{k_{j+1}}>q_{k_j}$. Define, for any $j\in\mathbb{N}$, $a_j\coloneqq\lceil{\log{q_{k_j}}/\log{2}}\rceil+1-k_j\geq1$, where $\lceil\,\rceil$ is the ceiling function. Then, for any $j\in\mathbb{N}$ and $q\in S_{k_j+a_j}\subset S_{k_j}$,  $q>q_{k_j}$ and implying
\begin{align}
    \label{eq: good k q}
    \frac{1}{q^{\overline{\lambda}_\psi+\varepsilon}}<\psi(q).
\end{align}
Thus, for any $j\in\mathbb{N}$ and $d_1,\ldots,d_{k_j+a_j}\in\mathbb{N}\setminus\{1\}$,
\begin{align*}
    \left(\frac{P}{Q},\frac{P}{Q}+\frac{1}{Q^{1+\overline{\lambda}_\psi+\varepsilon}}\right)
    \subset
    \left(\frac{P}{Q},\frac{P}{Q}+\frac{\psi(Q)}{Q}\right),
\end{align*}
where $Q=d_1(d_1-1)\cdots d_{k_j+a_j}\in S_{k_j+a_j}$ and $P=[d_1,\ldots,d_{k_j+a_j}]Q$. Thus,
\begin{align}
    \label{eq: Cylinder covered by blow-up}
    \left(\frac{P}{Q},\frac{P}{Q}+\frac{1}{Q}\right]
    \subset
    \left(\frac{P}{Q},\frac{P}{Q}+\frac{1}{Q^{1+\overline{\lambda}_\psi+\varepsilon}}\right)^{1/(1+\overline{\lambda}_\psi+\varepsilon)}
    \subset
    \left(\frac{P}{Q},\frac{P}{Q}+\frac{\psi(Q)}{Q}\right)^{1/(1+\overline{\lambda}_\psi+\varepsilon)}
\end{align}
where $B^s$ denotes the interval with the same centre as $B$ and with a radius equal to the radius of $B$ raised to the power of $s$. By Proposition \ref{prop: Luroth-Dirichlet} and \eqref{eq: Cylinder covered by blow-up}, for any $j\in\mathbb{N}$,
\begin{align}
    \label{eq: cover the whole}
    (0,1]
    \subset
    \bigcup_{d_1,\ldots,d_{k_j+a_j}\in\mathbb{N}\setminus\{1\}}
    \left(\frac{P}{Q},
    \frac{P}{Q}+\frac{\psi(Q)}{Q}\right)
    ^{1/(1+\overline{\lambda}_\psi+\varepsilon)}.
\end{align}
Let $((P_{k_i},Q_{k_i},d_{k_i}))_{i\in\mathbb{N}}$ be the sub-re-enumeration of $S=((P_k,Q_k,d_k))_{k\in\mathbb{N}}$ in \eqref{eq: S re-enumeration} so that for any $i\in\mathbb{N}$ and $q=Q_{k_i}$, \eqref{eq: good k q} is satisfied. Define a sequence of intervals $(B_{\psi,\varepsilon,i})_{i\in\mathbb{N}}$ by, for any $i\in\mathbb{N}$, 
\begin{align*}
    B_{\psi,\varepsilon,i}
    \coloneqq\left(\frac{P_{k_i}}{Q_{k_i}},\frac{P_{k_i}}{Q_{k_i}}+\frac{1}{{Q_{k_i}}^{1+\overline{\lambda}_\psi+\varepsilon}}\right)
    \cap\left(\frac{P_{k_i}}{Q_{k_i}},\frac{P_{k_i}}{Q_{k_i}}+\frac{1}{(d_{k_i}-1)Q_{k_i}}\right].
\end{align*}
Note that, for any $\varepsilon>0$,
\begin{align*}
    \limsup_{j\to+\infty}B_{\psi,\varepsilon,j}
    \subset
    L(\psi).
\end{align*}
By a straightforward computation, for any $i\in\mathbb{N}$,
\begin{align*}
    \left(\frac{P_{k_i}}{Q_{k_i}},\frac{P_{k_i}}{Q_{k_i}}+\frac{1}{Q_{k_i}}\right)\subset{B_{\psi,\varepsilon,j}}^{1/(1+\overline{\lambda}_\psi+\varepsilon)}
    .
\end{align*}

By \eqref{eq: cover the whole}, for any $j\in\mathbb{N}$, the intervals at the ${(k_j+a_j)}$-th level cover $(0,1]$. Thus \begin{align*}
    (0,1]\cap\limsup_{i\to+\infty}{B_{\psi,\varepsilon,i}}^{1/(1+\overline{\lambda}_\psi+\varepsilon)}=(0,1];
\end{align*}
in particular,
\begin{align*}
    \mathcal{L}\left([0,1]\cap 
    \limsup_{i\to+\infty}{B_{\psi,\varepsilon,i}}^{1/(1+\overline{\lambda}_\psi+\varepsilon)}\right)=1.
\end{align*}
By the sub-re-enumeration on $S$, $\lim_{i\to+\infty}Q_{k_i}=+\infty$ and consequently $\lim_{i\to+\infty}\operatorname{diam}{B_{\psi,\varepsilon,i}}=0$. By applying Proposition~\ref{prop: AMTP} and the monotonicity of Hausdorff dimension,
\begin{align*}
    \dim{L(\psi)
    }\geq\dim{\left(\limsup_{j\to+\infty}B_{\psi,\varepsilon,j}\right)}
    \geq\frac{1}{1+\overline{\lambda}_\psi+\varepsilon}.
\end{align*}
Since $\varepsilon>0$ is arbitrary, the lower bound is established.

The proof of Theorem~\ref{thm: DL 4} is completed.
\end{proof}

\subsection{Proof of Theorem~\ref{thm: DL 5}}

Theorem~\ref{thm: DL 5} presents a counterexample, and consequently suggests a reformulation, to the conjecture stated by Tan--Zhou \cite{tan2021approximation}. The proof constructs the function that takes small values at integers that are highly divisible by 2, and consequently does not satisfy the monotonicity condition.

\begin{proof}[Proof of Theorem~\ref{thm: DL 5}]
Let $\tau\geq0$. Define $\psi\coloneqq\psi_{\tau}:\mathbb{N}\to(0,1]$ by, for any $q\in\mathbb{N}$,
\begin{align*}
    \psi(q)=\frac{1}{q^{\tau+\nu_2(q)}},
\end{align*}
where $\nu_2$ is the 2-adic valuation. For any $n\in\mathbb{N}$, if $n>\tau+\sqrt{\tau}$ then
\begin{align*}
    \psi\left(2^n+1\right)=\frac{1}{(2^n+1)^\tau}\geq\frac{1}{2^{(\tau+1)n}}>\frac{1}{2^{(\tau+n)n}}=\psi(2^n).
\end{align*}
Thus, $\psi$ is not eventually non-increasing. By straightforward computation, $\underline{\tau}_\psi=\tau$ can be verified. It remains to prove that $\dim{L(\psi)}=0$.

For any $q,k\in\mathbb{N}$, if $2^k$ divides $q$ then $\nu_2(q)\geq k$ and
\begin{align}
    \label{eq: psi if 2k}
    0<\psi(q)\leq\frac{1}{q^{\tau+k}}.
\end{align}
Pick any $j\in\mathbb{N}$, $\rho>0$ and $N\in\mathbb{N}$. Suppose $N>-\log{\rho}/\log{2}$. The following forms a $\rho$-cover of $L(\psi)$:
\begin{align}
    \label{eq: rho cover DL 5}
    \bigcup_{n=N}^{\infty}\bigcup_{(d_1,\ldots,d_n)\in(\mathbb{N}\setminus\{1\})^n}\left(\frac{P_n}{Q_n},\frac{P_n}{Q_n}+\frac{\psi(Q_n)}{Q_n}\right),
\end{align}
where for any $n\in\mathbb{N}$, $Q_n\coloneqq d_1(d_1-1)\cdots d_{n-1}(d_{n-1}-1)d_n\geq2^n$ and $P_n\coloneqq [d_1,d_2,\ldots,d_n]Q_n$. Pick any $s>1/(1+\tau+j)$. Define
\begin{align*}
    r_{\tau,j,s}
    &\coloneqq\sum_{d=2}^\infty\left(\frac{1}{d(d-1)}\right)^{(1+\tau+j)s}\in(0,1); \\
    C_{\tau,j,s}
    &\coloneqq\sum_{d=2}^\infty\frac{1}{{d}^{(1+\tau+j)s}}<+\infty.
\end{align*}

Pick any $\varepsilon>0$. Without loss of generality, $N>j$ and
\begin{align*}
    C_{\tau,j,s}\sum_{n=N}^{\infty}{r_{\tau,j,s}}^{n-1}<\varepsilon.
\end{align*}
Note that for any $n\in\mathbb{N}$ and $d_1,d_2,\ldots,d_n\in\mathbb{N}\setminus\{1\}$, if $n\geq N$ then $n>j$ and $2^j\mid d_1(d_1-1)\cdots d_{n-1}(d_{n-1}-1)d_n$. By applying \eqref{eq: psi if 2k}, the $s$-Hausdorff measure for the $\rho$-cover \eqref{eq: rho cover DL 5} is obtained as
\begin{align*}
    \mathcal{H}^s_{\rho}(L(\psi))
    &\leq
    \sum_{n=N}^{\infty}\sum_{(d_1,\ldots,d_n)\in(\mathbb{N}\setminus\{1\})^n}\left(\frac{\psi(d_1(d_1-1)\cdots d_n)}{d_1(d_1-1)\cdots d_n}\right)^s \\
    &\leq\sum_{n=N}^{\infty}\sum_{(d_1,\ldots,d_n)\in(\mathbb{N}\setminus\{1\})^n}\left(\frac{1}{(d_1(d_1-1)\cdots d_n)^{1+\tau+j}}\right)^s \\
    &=\sum_{n=N}^{\infty}\left(\sum_{d=2}^\infty\left(\frac{1}{d(d-1)}\right)^{(1+\tau+j)s}\right)^{n-1}\sum_{d=2}^\infty\frac{1}{{d}^{(1+\tau+j)s}} \\
    &=C_{\tau,j,s}\sum_{n=N}^{\infty}{r_{\tau,j,s}}^{n-1}<\varepsilon.
\end{align*}
Hence, $\mathcal{H}^s(L(\psi))=0$ and consequently for any $j\in\mathbb{N}$, 
\begin{align*}
    \dim{L(\psi)}\leq\frac{1}{1+\tau+j}.
\end{align*}

The proof of Theorem~\ref{thm: DL 5} is completed.
\end{proof}

\subsection{Proof of Theorem~\ref{thm: DL 6}}

Theorem~\ref{thm: DL 6} provides a partial result to the revised conjecture. By Theorem~\ref{thm: DL 3}, it suffices to prove the following lower bound:
\begin{align*}
    \dim{L(\psi)}\geq\frac{1}{1+\underline{\tau}_\psi}.
\end{align*}
The idea is to apply the divergent case of Theorem~\ref{thm: DL 1} to deduce that $\mathcal{L}(\theta_s)=1$, for the auxiliary function $\theta_s$ defined in \eqref{eq: theta s}. By applying Proposition~\ref{prop: AMTP} again, the lower bound is then established.

By the assumption $\underline{\tau}_{\psi}>1$, the set $L(\psi)$ can be expressed as a $\limsup$ set of intervals. Proposition~\ref{prop: good tau limsup} can be compared with Proposition~\ref{prop: Luroth, limsup}.
\begin{prop}\label{prop: good tau limsup}
    Let $\psi:\mathbb{N}\to(0,1]$. Suppose $\underline{\tau}_{\psi}>1$. Then
    \begin{align*}
        L(\psi)=\limsup_{k\to+\infty}\left(\frac{P_k}{Q_k},\frac{P_k}{Q_k}+\frac{\psi(Q_k)}{Q_k}\right),
    \end{align*}
    where $((P_k,Q_k,d_k))_{k\in\mathbb{N}}$ is the re-enumeration of $S$ in \eqref{eq: S re-enumeration}.
\end{prop}
\begin{proof}
    The inclusion $\subset$ follows directly from the definition of $S$ and \eqref{eq: L tau better}. It remains to prove the reverse inclusion $\supset$.

    By the assumption that $\underline{\tau}_{\psi}>1$, there exists $q_0\in\mathbb{N}$ such that for any $q\in\mathbb{N}$, if $q>q_0$ then 
    \begin{align}
        \label{eq: tau>1}
        0<\psi(q)<\frac{1}{q}.
    \end{align}
    Pick any $x\in\limsup_{k\to+\infty}({P_k}/{Q_k},{P_k}/{Q_k}+{\psi(Q_k)}/{Q_k})$. By limit superior, there exists a strictly increasing sequence $(k_j)_{j\in\mathbb{N}}$ of positive integers such that for any $j\in\mathbb{N}$, $x\in({P_{k_j}}/{Q_{k_j}},{P_{k_j}}/{Q_{k_j}}+{\psi(Q_{k_j})}/{Q_{k_j}})$.
    Since $(Q_k)_{k\in\mathbb{N}}$ is non-decreasing, $\lim_{k\to\infty}Q_k=+\infty$ and there exists $j_0\in\mathbb{N}$ such that $Q_{k_{j_0}}>q_0$. For any $j\in\mathbb{N}$, if $j>j_0$ then $Q_{k_j}\geq Q_{k_{j_0}}>q_0$ and \eqref{eq: tau>1} implies that 
    \begin{align*}
        0<x-\frac{P_{k_j}}{Q_{k_j}}<\frac{\psi(Q_{k_j})}{Q_{k_j}}<\frac{1}{{Q_{k_j}}^2}<\frac{1}{(d_{k_j}-1)Q_{k_j}}.
    \end{align*}
    By Proposition~\ref{Luroth, supset}, $P_{k_j}/Q_{k_j}$ is a L\"uroth convergent of $x$.
\end{proof}

\begin{proof}[Proof of Theorem~\ref{thm: DL 6}]
By the assumption that $\underline{\tau}_{\psi}>1$, there exists $q_0\in\mathbb{N}$ such that for any $q\in\mathbb{N}$, if $q\geq q_0$ then \eqref{eq: tau>1} holds. Thus,
\begin{align*}
    \sum_{q=q_0}^{\infty}-\frac{\theta_s(q)\log{\theta_s(q)}}{q}
    &=\sum_{q=q_0}^{\infty}-\frac{q^{1-s}\psi^s(q)\log{\left(q^{1-s}\psi^s(q)\right)}}{q}\\
    &\geq \sum_{q=q_0}^{\infty}\left(\frac{\psi(q)}{q}\right)^s\log{q} \\
    &=+\infty
    .
\end{align*}
By the divergent case of Theorem~\ref{thm: DL 1},
\begin{align}
    \label{eq: thm 1 gives full measure}
    \mathcal{L}(L(\theta_s))=1.
\end{align} 

By the assumption $\underline{\tau}_{\psi}>1$ and Proposition~\ref{prop: good tau limsup}, 
\begin{align*}
    L(\psi)=\limsup_{k\to+\infty}\left(\frac{P_k}{Q_k},\frac{P_k}{Q_k}+\frac{\psi(Q_k)}{Q_k}\right)
    .
\end{align*}
By a straightforward computation,
\begin{align*}
    L(\theta_s)
    =\limsup_{k\to+\infty}\left(\frac{P_k}{Q_k},\frac{P_k}{Q_k}+\frac{\theta_s(Q_k)}{Q_k}\right)
    \subset\limsup_{k\to+\infty}\left(\frac{P_k}{Q_k},\frac{P_k}{Q_k}+\frac{\psi(Q_k)}{Q_k}\right)^s
    .
\end{align*}
By \eqref{eq: thm 1 gives full measure} and the monotonicity of Lebesgue measure,
\begin{align*}
    \mathcal{L}{\left([0,1]\cap\limsup_{k\to+\infty}\left(\frac{P_k}{Q_k},\frac{P_k}{Q_k}+\frac{\psi(Q_k)}{Q_k}\right)^s\right)}=1.
\end{align*}
By $\lim_{k\to+\infty}Q_k=+\infty$, $\lim_{k\to+\infty}\operatorname{diam}{({P_k}/{Q_k},{P_k}/{Q_k}+{\psi(Q_k)}/{Q_k})}=0$ is obtained. By applying Proposition~\ref{prop: AMTP},
\begin{align*}
    \dim(L(\psi))
    =
    \dim{\left(\limsup_{k\to+\infty}\left(\frac{P_k}{Q_k},\frac{P_k}{Q_k}+\frac{\psi(Q_k)}{Q_k}\right)\right)}
    \geq s=\frac{1}{1+\underline{\tau}_{\psi}}.
\end{align*}

The proof of Theorem~\ref{thm: DL 6} is completed.
\end{proof}

\bibliographystyle{siam}
\bibliography{name}

\end{document}